\documentclass[12pt]{article}
\usepackage{amsmath,amssymb,amsbsy,amsfonts,amsthm}

\begin{document}

\newtheorem{theorem}{Theorem}
\newtheorem{lemma}[theorem]{Lemma}
\newtheorem{cor}[theorem]{Corollary}
\newtheorem{prop}[theorem]{Proposition}

\newcommand{\comm}[1]{\marginpar{%
\vskip-\baselineskip 
\raggedright\footnotesize
\itshape\hrule\smallskip#1\par\smallskip\hrule}}

\def\cA{{\mathcal A}}
\def\cB{{\mathcal B}}
\def\cC{{\mathcal C}}
\def\cD{{\mathcal D}}
\def\cE{{\mathcal E}}
\def\cF{{\mathcal F}}
\def\cG{{\mathcal G}}
\def\cH{{\mathcal H}}
\def\cI{{\mathcal I}}
\def\cJ{{\mathcal J}}
\def\cK{{\mathcal K}}
\def\cL{{\mathcal L}}
\def\cM{{\mathcal M}}
\def\cN{{\mathcal N}}
\def\cO{{\mathcal O}}
\def\cP{{\mathcal P}}
\def\cQ{{\mathcal Q}}
\def\cR{{\mathcal R}}
\def\cS{{\mathcal S}}
\def\cT{{\mathcal T}}
\def\cU{{\mathcal U}}
\def\cV{{\mathcal V}}
\def\cW{{\mathcal W}}
\def\cX{{\mathcal X}}
\def\cY{{\mathcal Y}}
\def\cZ{{\mathcal Z}}

\def\C{\mathbb{C}}
\def\F{\mathbb{F}}
\def\K{\mathbb{K}}
\def\Z{\mathbb{Z}}
\def\R{\mathbb{R}}
\def\Q{\mathbb{Q}}
\def\N{\mathbb{N}}

\def\({\left(}
\def\){\right)}
\def\[{\left[}
\def\]{\right]}
\def\<{\langle}
\def\>{\rangle}

\def\e{e}

\def\eq{\e_q}
\def\eT{\e_T}

\def\fl#1{\left\lfloor#1\right\rfloor}
\def\rf#1{\left\lceil#1\right\rceil}
\def\mand{\qquad\mbox{and}\qquad}

\title{\bf Some mixed character sums}

\date{ }
\author{ 
{\sc   Bryce Kerr } \\
{Department of Pure Mathematics} \\
{University of New South Wales} \\
{Sydney, NSW 2052, Australia} \\
{\tt  b.kerr@student.unsw.edu.au}}
\date{ }

\date{}

\maketitle
\begin{abstract}
In this paper we consider a variety of mixed character sums. In particular we extend a bound of Heath-Brown and Pierce to the case of squarefree modulus, improve on a result of Chang for mixed sums in finite fields, we show in certain circumstances we may improve on some results of Pierce for multidimensional mixed sums and we extend a bound for character sums with products of linear forms to the setting of mixed sums. 
\end{abstract}


\section{Introduction}
Let $q$ be an integer, $\chi$ be a primitive multiplicative character$\mod q$ and let $F$ be a polynomial of degree $d$ with real coefficients. We consider a variety of character sums mixed with terms of the form $e^{2\pi i F(n)}$. The simplest example of such sums are given by
\begin{equation}
\label{eq:Sdef}
\sum_{M<n\le N+M}\chi(n)e^{2\pi iF(n)}.
\end{equation}
For $q$ prime, these sums were first studied by Enflo~\cite{Enflo} who outlines an argument which gives the bound
$$|S(\chi,F)|\le  N^{1-1/2^dr}q^{(r+1)/2^{d+2}r^2},$$
for integer $r\ge 1$ and $N\le H^{3/4+1/4r}$, which is nontrivial provided $H>q^{1/4+o(1)}$  (see~\cite[Theorem~1.1]{HBP}). This bound was improved by Chang~\cite{Chang} who showed that
\begin{align}
\label{eq:Chang1}
|S(\chi,F)|\ll  Nq^{-\varepsilon},
\end{align}
when $N\ge q^{1/4+\delta}$ and
$$\varepsilon=\frac{\delta^2}{4(1+2\delta)(2+(d+1)^2}.$$
In the same paper, Chang also considered a generalisation of the sums~\eqref{eq:Sdef} to arbitrary finite fields. More specifically, let $q$ be prime, $n$ an integer, $\chi$  and $\psi$ multiplicative and additive characters of  $\mathbb{F}_{q^{n}}$ respectivley and let $F$ be a polynomial of degree $d$ with coefficients in $\mathbb{F}_{q^{n}}$. Let $\omega_1,\dots,\omega_n$ be a basis for $\mathbb{F}_{q^{n}}$ over $\mathbb{F}_{q}$ and let $\cB$ denote the box
$$\cB=\{ \omega_1h_1+\dots+\omega_nh_n : 1\le h_i \le H\}.$$
 Then Chang showed that
\begin{equation}
\label{eq:Chang2}
\sum_{h\in \cB}\chi(h)\psi(F(h))\ll H^{n}q^{-\varepsilon},
\end{equation}
when $H\ge q^{1/4+\delta}$ and $$\varepsilon=\frac{\delta^2n}{4(1+2\delta)(2n+(d+1)^2)}.$$
Recently, Heath-Brown and Pierce have improved on the bound of Chang~\eqref{eq:Chang1} for prime fields showing that, subject to some conditions on $r$ related to Vinogradov's mean value theorem, we have
\begin{equation}
\label{eq:HBP}
\left|\sum_{M<n\le M+N}\chi(n)e^{2\pi i F(n)}\right|\le N^{1-1/r}q^{(r+1-d(d+1)/2)/4r(r-d(d+1)/2)},
\end{equation}
which can be compared directly with the result of Chang by noting that for small $\delta$ and $N\ge q^{1/4+\delta},$ we have
$$ N^{1-1/r}q^{(r+1-d(d+1)/2)/4r(r-d(d+1)/2)}\le Nq^{-\varepsilon},$$
where $\varepsilon$ behaves like (see~\cite[Section~4.2]{HBP})
$$\left(\frac{2\delta}{1+\sqrt{1+2d(d+1)\delta}}\right)^2.$$
Pierce has also considered a multidimensional version of the sums~\eqref{eq:Sdef}. Let $q_1,\dots,q_n$ be primes, $\chi_i$ a multiplicative character$\mod q_i$ and  $F$ a polynomial of degree $d$ in $n$ variables. In~\cite{Pierce}  Pierce has given a number of different bounds for sums of the form
\begin{equation}
\label{eq:Pierce}
\sum_{N_i<h_i\le N_i+H_i}\chi_1(h_1)\dots\chi_n(h_n)e^{2\pi i F(h_1,\dots,h_n)},
\end{equation}
and in the same paper Pierce also mentioned the following problem: Let $L_1,\dots,L_n$ be $n$ linear forms in $n$ variables which are linearly independent$\mod q$ and let $F$ be a polynomial of degree $d$ in $n$ variables. Then consider giving an upper bound for the sums
\begin{equation}
\label{eq:mlm}
\sum_{1\le h_i\le H}\chi(\prod_{j=1}^{n}L_j(h_1,\dots,h_n))e^{2\pi i F(h_1,\dots,h_n)}.
\end{equation}
The  sums~\eqref{eq:mlm} without the factor $e^{2\pi i F(h_1,\dots,h_n)}$ were first considered by Burgess~\cite{Bur3} whose bound was later improved in general by Bourgain and Chang~\cite{BC}. \\

In this paper we consider giving bounds for a variety of mixed character sums. We first consider the problem of extending the bound of Heath-Brown and Pierce~\eqref{eq:HBP} to squarefree modulus. The main obstacle in doing this is bounding the double mean value
$$\int_{0}^{1}\dots \int_{0}^{1}\sum_{\lambda=1}^{q}\left| \sum_{1\le v \le V}\beta_v\chi(\lambda+v)e^{2\pi i (\alpha_1 v+\dots+\alpha_d v^{d})} \right|^{2r}d\alpha_1\dots d\alpha_d,$$
which for the case of prime modulus, as done by Heath-Brown and Pierce~\cite{HBP}, relies on the Weil bounds for complete sums and Vinogradov's mean value theorem. For the case of squarefree modulus, we can use the Chinese remainder theorem, as done by Burgess~\cite{Bur1} for pure sums, so that we may apply the Weil bounds, although there are extra complications in incorporating bounds for Vinogradov's mean value theorem. Doing this we end up with a bound weaker than for prime modulus, although in certain cases we can get something just as sharp, in particular when $q$ does not have many prime factors. \\

We give an improvement on the bound~\eqref{eq:Chang1} of Chang for boxes over finite fields. We deal with the factor $\psi(F(h))$ in a similar fashion to the case of squarefree modulus. Our argument also relies on  Konyagin's bound on the multiplicitive energy of boxes in finite fields~\cite{Kon}, Vinogradov's mean value theorem and the Weil bounds for complete sums.   \\

 We show in certain cases we may improve on the results of Pierce for the sums~\eqref{eq:Pierce}. The argument of Pierce relies on  a multidimensional version of Vinogradov's mean value theorem due to Parsell, Prendiville and Wooley~\cite{PPW}.  Our improvement comes from averaging the sums~\eqref{eq:Pierce} in a suitable way so we end up applying the classical Vinogradov mean value theorem rather than the multidimensional version. Although in order to do this, we need the range of summation in each variable not to get too short and each of the $q_i$ in~\eqref{eq:Pierce} not to be too small, so our result is less general. \\

Finally, we consider the problem mentioned by Pierce in~\cite{Pierce}, of bounding the sums~\eqref{eq:mlm}. We obtain a result almost as strong as Bourgain and Chang~\cite{BC} for the case of pure sums. An essential part of our proof is the bound of Bourgain and Chang on multiplicative energy of systems of linear forms. \\

Our arguments use a different approach to that of Heath-Brown and Pierce~\cite{HBP}. The technique we use to deal with the factor $e^{2\pi i F(n)}$ can be though of an a generalisation of an idea of  Chamizo~\cite{Cham}, who gave a simple proof of the Burgess bound for incomplete Gauss sums, which in our case corresponds to mixed sums of degree 1. We also note that our method is capable of reproducing the results of Heath-Brown and Pierce~\cite{HBP}. We briefly indicate our technique for dealing with mixed sums in a general setting. Let $F(x,y)$ be a polynomial of degree $d$ with real coefficients, $\Phi(k,v)$ a sequence of complex numbers and consider the bilinear form
$$W=\sum_{1\le k \le K}\sum_{1\le v\le V}\gamma_k\beta_v\Phi(k,v)e^{2\pi i F(k,v)}.$$
We have
\begin{align*}
W&\le \sum_{1\le k \le K}|\gamma_k|\left|\sum_{1\le v\le V}\beta_v\Phi(k,v)e^{2\pi i F(k,v)} \right| \\
&\le \sum_{1\le k \le K}|\gamma_k|\max_{\alpha_1,\dots,\alpha_d \in \R}\left|\sum_{1\le v\le V}\beta_v\Phi(k,v)e^{2\pi i(\alpha_1v+\dots+\alpha_dv^d)} \right|.
\end{align*}
 For  $i=1,\dots, d$, we let $$\delta_i=\frac{1}{4V^i},$$ and define the functions $\phi_{i}(v)$ by
$$1=\phi_i(v)\int_{-\delta_i}^{\delta_i}e^{2\pi i x v^{i}}dx,$$
so that for $1<v\le V$ we have

$$\phi_i(v)=\frac{2\pi iv^i }{\sin(2\pi \delta_i v^i)} \ll \frac{1}{\delta_i }\ll V^i,$$

and
$$W\le \int_{-\delta_1}^{\delta_1}\dots\int_{-\delta_d}^{\delta_d}\sum_{1\le k \le K}|\gamma_k|\max_{\alpha_1,\dots,\alpha_d \in \R}\left|\sum_{1\le v\le V}\beta'_v\Phi(k,v)e^{2\pi i((\alpha_1+x_1)v+\dots+(\alpha_d+x_d)v^d)} \right|d{\mathbf{x}},$$
where 
$$\beta'_v=\beta_v\prod_{i=1}^{d}\phi_i(v).$$
Applying H\"{o}lder's inequality gives
\begin{align*}
W^{2r}&\le V^{-(2r-1)d(d+1)/2}\left(\sum_{1\le k \le K}|\gamma_k|^{2r/(2r-1)} \right)^{2r-1} \\ &  \times
\left(\sum_{1\le k \le K}\int_{-\delta_1}^{\delta_1}\dots\int_{-\delta_d}^{\delta_d}\max_{\alpha_1,\dots,\alpha_d \in \R}\left|\sum_{1\le v\le V}\beta'_v\Phi(k,v)e^{2\pi i((\alpha_1+x_1)v+\dots+(\alpha_d+x_d)v^d)} \right|^{2r}d\mathbf{x}\right).
\end{align*}
By extending the range of integration we may remove the condition $\max_{\alpha_1,\dots,\alpha_d \in \R}$, since
 \begin{align*}
& \sum_{1\le k \le K}\int_{-\delta_1}^{\delta_1}\dots\int_{-\delta_d}^{\delta_d}\max_{\alpha_1,\dots,\alpha_d \in \R}\left|\sum_{1\le v\le V}\beta'_v\Phi(k,v)e^{2\pi i((\alpha_1+x_1)v+\dots+(\alpha_d+x_d)v^d)} \right|^{2r}d\mathbf{x} \\ & \quad \quad \quad \ll
\sum_{1\le k \le K}\int_{[0,1]^{d}}\left|\sum_{1\le v\le V}\beta'_v\Phi(k,v)e^{2\pi i(x_1v+\dots+x_dv^d)} \right|^{2r}d\mathbf{x}.
\end{align*}
At this point we may try and estimate the last double mean value by combining Vinogradov's mean value theorem with techniques for estimaing the sum 
$$\sum_{1\le k \le K}\left|\sum_{1\le v\le V}\beta'_v\Phi(k,v) \right|^{2r},$$
or we may note that for some $\beta''_v$ we have
\begin{align*}
& \sum_{1\le k \le K}\int_{[0,1]^{d}}\left|\sum_{1\le v\le V}\beta'_v\Phi(k,v)e^{2\pi i(x_1v+\dots+x_dv^d)} \right|^{2r}d\mathbf{x} \\ & \quad \quad \quad \quad \quad \quad \quad \quad  \le \sum_{1\le k \le K}\left|\sum_{1\le v\le V}\beta''_v\Phi(k,v) \right|^{2r}.
\end{align*}
Although our approach is different to that of Heath-Brown and Pierce, we also rely on bounds for Vinogradov's mean value theorem. For integers $r,d,V,$ we let $J_{r,d}(V)$ denote the number of solutions to the system of equations
$$v_1^{i}+\dots+v_{r}^{i}=v_{r+1}^{i}+\dots+v_{2r}^{i}, \quad 1\le i \le d, \quad 1\le v_j \le V.$$ 
Then it is conjectured that for any $r,d,V$ we have 
\begin{equation}
\label{eq:mvtc}
J_{r,d}(V)\le (X^{r}+X^{2r-d(d+1)/2})X^{o(1)}.
\end{equation}
Recently, Wooley~\cite{Wo1,Wo2} has made siginificant progress towards this conjecture. We state our main results in terms of the smallest integer $r_d$ such that we have a bound 
$$J_{r,d}(V)\le X^{2r-d(d+1)/2+o(1)},$$
valid for all $r\ge r_d$. Our results may then be combined with those of Wooley~\cite{Wo1,Wo2} to give admissible values of $r$ for which our bounds hold.
\section{Main Results}
In what follows, $r_d$ will be defined as in the introduction. We also let $D=d(d+1)/2$. Our first two Theorems consider mixed sums to squarefree modulus. 
\begin{theorem}
\label{thm:main1}
Let $q$ be squarefree and $\chi$ a primitive character$\mod q$. Let $M,N,r$ be integers such that $r\ge r_d$ and $N\le q^{1/2+1/4(r-D/2)}$. For any polynomial $F(x)$ of degree $d$ with real coefficients, we have
$$\left|\sum_{M<n\le M+N}\chi(n)e^{2\pi i F(n)}\right|\le  N^{1-1/r}q^{1/4r+D/8r(r-D/2)+1/4r(r-D/2)+o(1)}. $$
\end{theorem}
 Theorem~\ref{thm:main1} is slightly worse than the bound of Heath-Brown and Pierce~\eqref{eq:HBP} for prime modulus. Although in certain cases we can get something almost as strong (except for the conditions on $r$).
\begin{theorem}
\label{thm:main2}
Let let $s$ be an integer, $q$ be squarefree with at most $s$ prime factors and $\chi$ a primitive character$\mod q$. Let $M,N,r$ be integers with $N\le q^{1/2+1/4(r-D)}$ and $r\ge r_d+s+1$. For any polynomial $F(x)$ of degree $d$ with real coefficients, we have 
$$\left|\sum_{M<n\le M+N}\chi(n)e^{2\pi i F(n)}\right|\le   N^{1-1/r}q^{(r+1-D)/4r(r-D)+o(1)}. $$
\end{theorem}
Our  next Theorem improves the bound of Chang for mixed sums in finite fields~\cite{Chang}. Before we state our result we introduce some notation. Let $\omega_1,\dots,\omega_n$ be a basis for $\F_{q^{n}}$ over $\F_q$ and let $F$ be a polynomial of degree $d$ in $n$ variables with real coefficients. For $x\in \F_{q^n}$ we define $F(x)$ by 
$$F(x)=F(h_1,\dots,h_n),$$
where
$$x=h_1\omega_1+\dots+h_n\omega_n.$$
\begin{theorem}
\label{thm:main3}
Let $q$ be prime, $n$ an integer  and  $\chi$ be a multiplicative character of $\F_{q^n}$. Let $\omega_1,\dots,\omega_n$ be a basis for $\F_{q^n}$ as a vector space over $\F_q$. For integer $H$ let $\cB$ denote the box
$$\cB=\{ h_1\omega_1+\dots+h_n\omega_n :  0<h_i \le H \}.$$
Let  $F$ be a polynomial of degree $d$ in $n$ variables with real coefficients.
Then if $H\le q^{1/2}$ and $r\ge r_d$ we have
$$\left|\sum_{\mathbf{x}\in \cB}\chi(\mathbf{x})e^{2\pi i F(\mathbf{x})} \right| \le  (\#\cB)^{1-1/r}q^{n(r-D+1)/4r(r-D)+o(1)}.$$
\end{theorem}
We note that the sums in Theorem~\ref{thm:main3} are slightly more general than those considered by Chang~\cite{Chang}, since any additive character $\psi$ of $\F_{q^{n}}$ is of the form $$\psi(x)=e^{2\pi i \text{Tr}(ax)/q},$$ for some $a\in \F_{q^{n}}.$ \\

Our next Theorem improves on some results of Pierce~\cite{Pierce} in certain circumstances.
\begin{theorem}
\label{thm:main4}
Let $q_1,\dots,q_n$ be primes, which may not be distinct, and let $\chi_i$ be a multiplicative character$\mod q_i$. Let $F$ be a polynomial of degree $d$ in $n$ variables with real coefficients and  let  $\cB$ denote the box
$$\cB=\{ (h_1,\dots,h_n): M_i< h_i\le M_i+H_i\}.$$
For any integer $r\ge r_d$, if for each $i$ we have $q_i>q^{1/2(r-D)}$ and $q^{1/2(r-D)}\le H_i\le q_i^{1/2+1/4(r-D)},$  then we have
$$\left|\sum_{\mathbf{x}\in \cB}\chi_1(x_1)\dots \chi_n(x_n)e^{2\pi i F(\mathbf{x})} \right|\le (\#\cB)^{1-1/r}q^{(r-D+n)/4r(r-D)+o(1)},$$
where $q=q_1\dots q_n.$
\end{theorem}
Our final Theorem extends a bound of Bourgain and Chang~\cite{BC} to the setting of mixed character sums.
\begin{theorem}
\label{thm:main5}
Let $q$ be prime and $\chi$ a multiplicative character$\mod q$. Let 
$L_1,\dots,L_n$ be linear forms with integer coefficients in $n$ variables which are linearly independent$\mod q$. Let $\cB$ denote the box
$$\cB=\{ (h_1,\dots,h_n): 1< h_i\le H\},$$
and  let  $F$ be a polynomial of degree $d$ in $n$ variables with real coefficients. Then if $H\le q^{1/2}$ and $r\ge r_d$ we have
$$\left|\sum_{\mathbf{x}\in \cB}\chi\left(\prod_{i=1}^{n}L_i(\mathbf{x}) \right)e^{2\pi i F(\mathbf{x})} \right|\le (\#\cB)^{1-1/r}q^{n(r-D+1)/4r(r-D)+o(1)} .$$
\end{theorem}

\section{Preliminary results}
The following can be thought of a multidimensional version of a technique from the proof of~\cite[Theorem~1]{FI}.
\begin{lemma}
\label{lem:smooth}
Let $\mathbf{n}=(n_1,\dots,n_r)$ and $G(\mathbf{n})$ be any complex valued function on the integers. Let $\cB$ and $\cB_0$ denote the boxes
$$\cB=\{ (n_1,\dots,n_r)\in \mathbb{Z}^k: 1\le n_i \le N_i, 1\le i \le r  \},$$
$$\cB_0=\{ (n_1,\dots,n_r)\in \mathbb{Z}^k: -N_i\le n_i \le N_i, 1\le i \le r  \}.$$
Let $U_1,\dots,U_n$ and $V$ be positive integers such that $U_iV\le N_i$ and let $\cU\subset \mathbb{Z}^{r}$ be any set such that if $(u_1,\dots,u_n)\in \cU$ then $1\le u_i\le U_i$. 
Then for some $\alpha \in \mathbb{R}$ we have
$$\left|\sum_{\mathbf{n}\in \cB}G(\mathbf{n}) \right|\ll \frac{\log{N_1}\dots \log{N_r}}{V\#\cU} \sum_{\mathbf{n}\in\cB_0}\sum_{\mathbf{u}\in\cU}\left|\sum_{1\le v \le V}G(\mathbf{n}+v\mathbf{u})e^{2\pi i \alpha v}\right|. $$
\end{lemma}
\begin{proof}
For $1\le i \le r$ let 
$$f_i(x)=\begin{cases}\min(x-1,1,N_i-x), \quad \text{if} \quad 1 \le x \le N_i, \\ 
0, \quad \quad \quad \quad \quad \quad \quad \quad \quad \quad \quad \quad \    \text{otherwise}, \end{cases}$$
and 
$$f(\mathbf{x})=\prod_{i=1}^{r}f_i(x_i).$$
 Let  $g(\mathbf{y})$ denote the Fourier transform of $f$, so that
$$g(\mathbf{y})=\frac{1}{(2\pi i)^r}\int_{\mathbb{R}^k}f(\mathbf{x})e^{-2\pi i <\mathbf{x},\mathbf{y}>}d\mathbf{x}.$$
Integrating the above integral by parts in each dimension gives
\begin{equation}
\label{eq:inversionbound}
|g(\mathbf{y})|\ll \prod_{i=1}^{k}\min\left(N_i,\frac{1}{|y_i|},\frac{1}{|y_i|^2}\right) .
\end{equation}
For $\mathbf{u}\in\cU$ and $1\le v \le V$ we have
$$\sum_{\mathbf{n}\in \cB}G(\mathbf{n})=\sum_{\mathbf{n}\in \cB_0}f(\mathbf{n}+v\mathbf{u})G(\mathbf{n}+v\mathbf{u}),$$
hence by Fourier inversion
\begin{align*}
\sum_{\mathbf{n}\in \cB}G(\mathbf{n})&=\sum_{\mathbf{n}\in \cB_0}f(\mathbf{n}+v\mathbf{u})G(\mathbf{n}+v\mathbf{u}) \\
&= \frac{1}{(2\pi i)^r}\sum_{\mathbf{n}\in \cB_0}\int_{\mathbb{R}^{r}}g(\mathbf{y})G(\mathbf{n}+v\mathbf{u})e^{2\pi i<\mathbf{n}+v\mathbf{u},\mathbf{y}>}d\mathbf{y}.
\end{align*}
For $\mathbf{u}=(u_1,\dots,u_r)$ we let $|\mathbf{u}|=u_1\dots u_r$ and $\mathbf{u}^{-1}=(u_1^{-1},\dots,u_r^{-1})$. Then the  change of variable $\mathbf{y}=\mathbf{u}^{-1}\mathbf{x}$ in the above integral gives 
\begin{align*}
\sum_{\mathbf{n}\in \cB}G(\mathbf{n})= \frac{1}{(2\pi i)^r}\sum_{\mathbf{n}\in \cB_0}\int_{\mathbb{R}^{r}}\frac{1}{|\mathbf{u}|}g(\mathbf{u}^{-1}\mathbf{x})G(\mathbf{n}+v\mathbf{u})e^{2\pi i <\mathbf{n}+v\mathbf{u},\mathbf{u}^{-1}\mathbf{x}>}d\mathbf{x},
\end{align*}
so that averaging over $\mathbf{u}v$ with  $\mathbf{u}\in \cU$ and $1\le v \le V$ we get
\begin{align*}
\sum_{\mathbf{n}\in \cB}G(\mathbf{n}) &=\frac{1}{(2\pi i)^r\#\cU V}\sum_{\mathbf{n}\in\cB_0}\sum_{\mathbf{u}\in\cU}\sum_{1\le v \le V}\int_{\mathbb{R}^{r}} \\ & \quad \quad \quad \quad \frac{1}{|\mathbf{u}|}g(\mathbf{u}^{-1}\mathbf{x})G(\mathbf{n}+v\mathbf{u})e^{2\pi i <\mathbf{n}+v\mathbf{u},\mathbf{u}^{-1}\mathbf{x}>}d\mathbf{x},
\end{align*}
hence by~\eqref{eq:inversionbound}
\begin{align*}
\sum_{\mathbf{n}\in \cB}G(\mathbf{n})e^{2\pi i F(\mathbf{n})} & \ll \frac{1}{V\#\cU} \sum_{\mathbf{n}\in\cB_0}\sum_{\mathbf{u}\in\cU}\int_{\mathbb{R}^{r}} \frac{1}{|\mathbf{u}|}g(\mathbf{u}^{-1}\mathbf{x})\left|\sum_{1\le v \le V}G(\mathbf{n}+v\mathbf{u})e^{2\pi i <v,\mathbf{x}>}\right|d\mathbf{x} \\
&\ll \frac{1}{V\#\cU} \int_{\mathbb{R}^{r}} \prod_{i=1}^{r}\min\left(N_i,\frac{1}{|x_i|},\frac{U_i}{|x_i|^2} \right) \\ & \quad \quad \quad \quad \times \sum_{\mathbf{n}\in\cB_0}\sum_{\mathbf{u}\in\cU}\left|\sum_{1\le v \le V}G(\mathbf{n}+v\mathbf{u})e^{2\pi i <v,\mathbf{x}>}\right|d\mathbf{x} \\
&\ll \max_{\beta \in \mathbb{R}}\frac{1}{V\#\cU} \sum_{\mathbf{n}\in\cB_0}\sum_{\mathbf{u}\in\cU}\left|\sum_{1\le v \le V}G(\mathbf{n}+v\mathbf{u})e^{2\pi i \beta v}\right| \\
& \quad \quad \quad \quad \quad \times \prod_{i=1}^{r}\left(\int_{-\infty}^{\infty} \min\left(N_i,\frac{1}{|y|},\frac{U_i}{|y|^2} \right)dy\right).
\end{align*}
Since $U_i\le N_i$, we see that 
$$\prod_{i=1}^{r}\left(\int_{-\infty}^{\infty} \min\left(N_i,\frac{1}{|y|},\frac{U_i}{|y|^2} \right)dy\right)\ll \log{N_1}\dots \log{N_r},$$
and the result follows by letting $\alpha$ be defined by 
\begin{align*}
\max_{\beta \in \mathbb{R}} \sum_{\mathbf{n}\in\cB_0}\sum_{\mathbf{u}\in\cU}\left|\sum_{1\le v \le V}G(\mathbf{n}+v\mathbf{u})e^{2\pi i \beta v}\right|= \sum_{\mathbf{n}\in\cB_0}\sum_{\mathbf{u}\in\cU}\left|\sum_{1\le v \le V}G(\mathbf{n}+v\mathbf{u})e^{2\pi i \alpha v}\right|.  
\end{align*}
\end{proof}
\section{Mean value estimates}
We keep notation as in the introduction and we recall that $J_{r,d}(V)$ denotes the number of solutions to the system of equations
$$v_1^{i}+\dots+v_{r}^{i}=v_{r+1}^{i}+\dots+v_{2r}^{i}, \quad 1\le i \le d, \quad 1\le v_j \le V.$$ 

The following is due to Burgess and is a special case of~\cite[Lemma~7]{Bur1}, although since the statement of Burgress is weaker than what the argument implies, we reproduce the proof.
\begin{lemma}
\label{lem:charsumcubefree}
Let $q$ be squarefree, $\chi$  a primitive character$\mod q$, let $\mathbf{v}=(v_1,\dots,v_{2r})$ be a $2r$-tuple of integers such that at least $r+1$ of the $v_i$'s are distinct and let
$$A_i(\mathbf{v})=\prod_{\substack{j=1 \\ j\neq i}}^{2r}(v_i-v_j).$$
Then for any $1\le i \le 2r$ such that $A_i(\mathbf{v})\neq 0$ we have
\begin{equation}
\label{eq:completecharbound}
\sum_{\lambda=1}^{q}\chi\left(\frac{(\lambda+v_1)\dots(\lambda+v_r)}{(\lambda+v_{r+1})\dots(\lambda+v_{2r})} \right)\le (q,A_{i}(\mathbf{v}))^{1/2}q^{1/2+o(1)}.
\end{equation}
\end{lemma}
\begin{proof}
Let 
$$q=p_1\dots p_{k},$$
be the prime factorization of $q$, then by the Chinese remainder theorem there exists primitive characters 
$$\chi_j \mod p_j, \quad 1\le j \le k,$$
such that
$$\chi=\chi_1\dots \chi_{k},$$
and
\begin{align*}
\sum_{\lambda=1}^{q}\chi\left(\frac{(\lambda+v_1)\dots(\lambda+v_r)}{(\lambda+v_{r+1})\dots(\lambda+v_{2r})} \right)&=
\prod_{j=1}^{k}\left(\sum_{\lambda=1}^{p_j}\chi_j\left(\frac{(\lambda+v_1)\dots(\lambda+v_r)}{(\lambda+v_{r+1})\dots(\lambda+v_{2r})} \right) \right).
\end{align*}
We note that since at least $r+1$ of the $v_j$ are distinct there exists an $i$ such that $A_i(\mathbf{v})\neq 0$, hence from~\cite[Lemma~1]{Bur1} we have
\begin{equation}
\label{eq:completebound1}
\sum_{\lambda=1}^{p_j}\chi_j\left(\frac{(\lambda+v_1)\dots(\lambda+v_r)}{(\lambda+v_{r+1})\dots(\lambda+v_{2r})} \right)\ll 
(p_j,A_{i}(\mathbf{v}))^{1/2}p_j^{1/2},
\end{equation}
which by the above gives
$$\left|\sum_{\lambda=1}^{q}\chi\left(\frac{(\lambda+v_1)\dots(\lambda+v_r)}{(\lambda+v_{r+1})\dots(\lambda+v_{2r})} \right)\right|\le (q,A_{i}(\mathbf{v}))^{1/2}q^{1/2+o(1)}.$$

\end{proof}
The following will be used in the proof of Theorem~\ref{thm:main1}.
\begin{lemma}
\label{lem:mvsquarefree}
Let $q$ be squarefree, $\chi$ a primitive character$\mod q$, $\beta_v$ be a sequence of complex numbers with $|\beta_v|\le 1$ and let
\begin{align}
\label{eq:doublemeanvalue}
W=\int_{0}^{1}\dots \int_{0}^{1}\sum_{\lambda=1}^{q}\left| \sum_{1\le v \le V}\beta_v\chi(\lambda+v)e^{2\pi i (\alpha_1 v+\dots+\alpha_k v^{k})} \right|^{2r}d\alpha_1\dots d\alpha_k .
\end{align}
Then we have
$$W\le \left(qV^{r}+q^{1/2}J_{r,k}(V)^{1/2}V^r \right)q^{o(1)} .$$
\end{lemma}
\begin{proof}
Let $\cJ_{r,k}(V)$ denote the set of all $(v_1,\dots,v_{2r})$ such that 
$$v_1^{j}+\dots+v_r^{j}=v_{r+1}^{j}+\dots+v_{2r}^{j}, \quad 1\le j \le k, \quad 1\le v_i \le V,$$
then expanding the $2r$-th power in the definition of $W$ and interchanging summation and integration gives
\begin{align*}
W \le \sum_{(v_1,\dots,v_{2r})\in \cJ_{r,k}(V)}\left|\sum_{\lambda=1}^{q}\chi \left(\frac{(\lambda+v_1)\dots(\lambda+v_r)}{(\lambda+v_{r+1}\dots(\lambda+v_{2r})} \right)\right|.
\end{align*}
We break $\cJ_{r,k}(V)$ into sets $\cJ_{r,k}'(V)$ and $\cJ_{r,k}''(V)$, where
\begin{align*}
\cJ_{r,k}'(V)&=\{ (v_1,\dots,v_{2r})\in \cJ_{r,k}(V) : \text{at least} \  r+1 \  \text{of the} \  v_i's \   \text{are distinct}\}, \\
\cJ_{r,k}''(V)&=\{ (v_1,\dots,v_{2r})\in \cJ_{r,k}(V) : (v_1,\dots,v_{2r})\not \in \cJ_{r,k}'(V) \},
\end{align*} 
so that $\# \cJ_{r,k}''(V)\ll V^{r}$ and  by Lemma~\ref{lem:charsumcubefree} we have
\begin{align*}
W &\ll qV^{r}+q^{1/2+o(1)}\left(\sum_{(v_1,\dots,v_{2r})\in \cJ_{r,k}'(V)}\sum_{\substack{i=1 \\ A_i(\mathbf{v})\neq 0}}^{2r}(A_i(\mathbf{v}),q)^{1/2}\right) \\
&=qV^{r}+q^{1/2+o(1)}\left(\sum_{i=1}^{2r}\sum_{\substack{ (v_1,\dots,v_{2r})\in \cJ_{r,k}'(V) \\ A_i(\mathbf{v})\neq 0}}(A_i(\mathbf{v}),q)^{1/2}\right).
\end{align*}
For $1\le i \le 2r$ let  
$$W_i=\sum_{\substack{ (v_1,\dots,v_{2r})\in \cJ_{r,k}'(V) \\ A_i(\mathbf{v})\neq 0}}(A_i(\mathbf{v}),q)^{1/2},$$
so that by the Cauchy-Schwartz inequality we have
\begin{align*}
W_i&\le \left(\sum_{\substack{ (v_1,\dots,v_{2r})\in \cJ_{r,k}'(V) }}1\right)^{1/2}\left(\sum_{\substack{ (v_1,\dots,v_{2r}) \\ A_i(\mathbf{v})\neq 0}}(A_i(\mathbf{v}),q) \right)^{1/2} \\
&\le \# \cJ_{r,k}(V)^{1/2}\left(\sum_{\substack{ v_1,\dots,v_{2r} \\ A_i(\mathbf{v})\neq 0}}(A_i(\mathbf{v}),q) \right)^{1/2}.
\end{align*}
For the last sum, we have
\begin{align*}
\sum_{\substack{ (v_1,\dots,v_{2r}) \\ A_i(\mathbf{v})\neq 0}}(A_i(\mathbf{v}),q) \ll \sum_{d|q}d\sum_{\substack{A\neq 0 \\ d|A}}\sum_{\substack{v_1,\dots,v_{2r} \\  A_i(\mathbf{v})=A} }1.
\end{align*}
Considering the innermost sum, for fixed $A$ if $(v_1,\dots,v_{2r})$ are such that 
$$A_i(\mathbf{v})=A,$$
then since
$$A_i(\mathbf{v})=\prod_{\substack{j=1 \\ j\neq i}}^{2r}(v_i-v_j),$$
we see there are $q^{o(1)}$ choices for the numbers $(v_i-v_1), \dots, (v_i-v_{2r})$ and choosing $v_i$ determines $v_1,\dots,v_{2r},$ uniquley. Since there are $V$ choices for $v_i$ and each $A_i(\mathbf{v})\ll V^{2r-1}$ we get
\begin{align*}
\sum_{\substack{ (v_1,\dots,v_{2r}) \\ A_i(\mathbf{v})\neq 0}}(A_i(\mathbf{v}),q) &\ll q^{o(1)}\sum_{d|q}d\sum_{\substack{1\le A \ll V^{2r-1} \\ d|A}}V \\
& \ll q^{o(1)}V^{2r}\sum_{d|q}1 \ll q^{o(1)}V^{2r},
\end{align*}
which gives
$$W\le \left(qV^{r}+q^{1/2}\#\cJ_{r,k}(V)^{1/2}V^r \right)q^{o(1)}.$$
\end{proof}
The following will be used in the proof of Theorem~\ref{thm:main2} and improves on Lemma~\ref{lem:mvsquarefree} provided the number of prime factors of $q$ is bounded.
\begin{lemma}
\label{lem:mvsquarefree1}
Let $s$ be an integer and let $q$ be squarefree such that the number of prime factors of $q$ is less than $s$. Let  $\chi$ a primitive character$\mod q$, $\beta_v$ be any sequence of complex numbers with $|\beta_v|\le 1$ and for $r\ge s+1$  let
\begin{align}
\label{eq:doublemeanvalue1}
W=\int_{0}^{1}\dots \int_{0}^{1}\sum_{\lambda=1}^{q}\left| \sum_{1\le v \le V}\beta_v\chi(\lambda+v)e^{2\pi i (\alpha_1 v+\dots+\alpha_k v^{d})} \right|^{2r}d\alpha_1\dots d\alpha_d. 
\end{align}
Then we have
$$W\le \left(qV^{r}+q^{1/2}J_{r-s-1,d}(V)V^{2s+2} \right)q^{o(1)} .$$
\end{lemma}
\begin{proof}
We keep the same notation from the proof of Lemma~\ref{lem:mvsquarefree}, so that following the same argument gives 
\begin{align*}
W\ll qV^{r}+q^{1/2+o(1)}\left(\sum_{i=1}^{2r}W_i\right),
\end{align*}
where 
\begin{equation}
\label{eq:Wdef}
W_i=\sum_{\substack{ (v_1,\dots,v_{2r})\in \cJ_{r,k}'(V) \\ A_1(\mathbf{v})\neq 0}}(A_i(\mathbf{v}),q)^{1/2}.
\end{equation}
We consider only $W_1$, the same argument applies to the remaining $W_i.$
Let $q=q_1\dots q_s$ be the prime factorization of $q$ and for each subset  $\cS \subseteq \{1,\dots,s\}$ we partition $\cS$ into
$2r-1$  sets 
\begin{equation}
\label{eq:Upartition}
\cS=\bigcup_{j=2}^{2r}U_j, \quad \text{where} \quad U_i \cap U_j = \emptyset \quad \text{if} \quad i\neq j,
\end{equation}
where some $U_j$ may be empty. We have
\begin{equation}
\label{eq:Wpart}
W_1\le \sum_{\cS\subseteq \{1,\dots,s\}}\sum_{U_2,\dots,U_{2r}}\sum_{\substack{ (v_1,\dots,v_{2r})\in \cJ_{r,k}'(V) \\ A_1(\mathbf{v})\neq 0 \\ (q,v_1-v_j)=\prod_{\ell \in U_j}q_\ell}}(A_1(\mathbf{v}),q)^{1/2},
\end{equation}
where the sum over $U_2,\dots,U_{2r}$ satisfies~\eqref{eq:Upartition}.
Hence it is sufficient to show that for fixed $\cS$ and fixed $U_2,\dots, U_{2r}$ satisfying~\eqref{eq:Upartition} we have
$$\sum_{\substack{ (v_1,\dots,v_{2r})\in \cJ_{r,k}'(V) \\ A_1(\mathbf{v})\neq 0 \\ (q,v_1-v_j)=\prod_{\ell \in U_j}q_\ell}}(A_1(\mathbf{v}),q)^{1/2}\le J_{r-s-1,k}(V)V^{2s+2}q^{o(1)}.$$
Considering values of $j$ such that $U_j\neq \emptyset$, each value of $v_1$ determines $v_j$ with $\ll V/\prod_{i\in U_j}q_i$ possibilities. Since there are are most $s$ values of  $j$  such that $U_j \neq \emptyset$, we may choose two sets $\cV_1,\cV_2$ such that
$$\cV_1 \subseteq \{1,\dots,r\}, \quad \#\cV_1=r-s-1,$$
$$\cV_2 \subseteq \{r+1,\dots,2r\}, \quad \#\cV_2=r-s-1,$$
and integers $\alpha_1,\dots,\alpha_k$ such that
$$\sum_{\substack{ (v_1,\dots,v_r)\in \cJ_{r,k}'(V) \\ A_1(\mathbf{v})\neq 0 \\ (q,v_1-v_j)=\prod_{i\in U_j}q_i}}(A_1(\mathbf{v}),q)^{1/2}\ll V^{2s+2}J(\cV_1,\cV_2,\alpha_1,\dots,\alpha_k),$$
where $J(\cV_1,\cV_2,\alpha_1,\dots,\alpha_k)$ denotes the number of solutions to the system of equations
$$\sum_{j\in \cV_1}v_j^{i}-\sum_{j\in \cV_2}v_j^{i}=\alpha_i, \quad 1\le i \le d, \quad 1\le v_j\le V.$$
Since $J(\cV_1,\cV_2,\alpha_1,\dots,\alpha_k)\le J_{r-s-1,d}(V)$, we see that 
$$\sum_{\substack{ (v_1,\dots,v_{2r})\in \cJ_{r,k}'(V) \\ A_1(\mathbf{v})\neq 0 \\ (q,v_1-v_j)=\prod_{i\in U_j}q_i}}(A_1(\mathbf{v}),q)^{1/2}\le V^{2s+2}J_{r-s-1,d}(V),$$
so that 
$$W_1\le V^{2s+2}J_{r-s-1,d}(V)q^{o(1)},$$
which completes the proof.
\end{proof}
\begin{lemma}
\label{lem:mvmulti}
Let $q_1,\dots,q_n$ be primes, $\chi_i$ a multiplicative character$\mod q_i$, $\beta_v$ be a sequence of complex numbers with $|\beta_v|\le 1$ and let
\begin{align}
\label{eq:doublemeanvalue4}
W=\int_{0}^{1}\dots \int_{0}^{1}\sum_{\substack{\lambda_i=1 \\ 1\le i \le n}}^{q_i}\left| \sum_{1\le v \le V}\beta_v\prod_{i=1}^{n}\chi_i(\lambda_i+v)e^{2\pi i (\alpha_1 v+\dots+\alpha_k v^{k})} \right|^{2r}d\alpha_1\dots d\alpha_k .
\end{align}
Then if $V\le q_{i}$ for each $i$ we have
$$W\le \left(qV^{r}+q^{1/2}J_{r,k}(V) \right)q^{o(1)},$$
where $q=q_1\dots q_n.$
\end{lemma}
\begin{proof}
With notation as in the proof of Lemma~\ref{lem:mvsquarefree}, following the same argument gives
\begin{align*}
W &\ll qV^{r}+\left(\sum_{(v_1,\dots,v_{2r})\in \cJ_{r,k}'(V)}\prod_{i=1}^{n}\left|\sum_{\lambda=1}^{q_i}\chi_i\left( \frac{(\lambda+v_1)\dots (\lambda+v_r)}{(\lambda+v_{r+1})\dots(\lambda+v_{2r})}\right) \right|\right).
\end{align*}
We claim that if $(v_1,\dots,v_{2r})\in \cJ_{r,k}'(V)$ then for each $1\le i \le n$ the function
$$\chi_i\left( \frac{(\lambda+v_1)\dots (\lambda+v_r)}{(\lambda+v_{r+1})\dots(\lambda+v_{2r})}\right),$$
 is not constant. Supposing for some $i$ this were false and letting $d$ denote the order of $\chi_i$, then this implies that the rational function
$$\frac{(\lambda+v_1)\dots (\lambda+v_r)}{(\lambda+v_{r+1})\dots(\lambda+v_{2r})},$$
is a $d$-th power$\mod q_i$, so that at most $r+1$ of the $v_1,\dots,v_{2r}$ are distinct$\mod q_i$, and since $V<q_i$ this implies that at most $r+1$ of the $v_1,\dots, v_{2r}$ are distinct, contradicting the definition of $\cJ_{r,k}'(V)$. Hence from the Weil bound for complete character sums~\cite[Theorem 2C', pg 43]{Sch} we have
$$\sum_{\lambda=1}^{q_i}\chi_i\left( \frac{(\lambda+v_1)\dots (\lambda+v_r)}{(\lambda+v_{r+1})\dots(\lambda+v_{2r})}\right)\ll q_{i}^{1/2},$$
provided $(v_1,\dots,v_{2r})\in \cJ_{r,k}'(V).$ Hence we get
$$W \ll qV^{r}+q^{1/2+o(1)}\#\cJ_{r,k}'(V)\le (qV^{r}+q^{1/2}J_{r,d}(V))q^{o(1)}.$$
\end{proof}
\begin{lemma}
\label{lem:mv2}
Let $q$ be prime, $n$ an integer and $\chi$ a multiplicative character of $\F_{q^{n}}$, $\beta_v$ any sequence of complex numbers satisfying $|\beta_v|\le 1$ and let
$$W=\int_{[0,1]^{d}}\sum_{\lambda \in \F_{q^n}}\left|\sum_{1\le v \le V}\beta_v \chi(\lambda+v)e^{2\pi i (\alpha_1 v_1+\dots+\alpha_d v^d)}\right|^{2r}d\alpha_1 \dots d\alpha_d.$$
Then for any integer $r$ we have
$$W\ll q^{n}V^{r}+q^{n/2}J_{r,d}(V).$$
\end{lemma}
\begin{proof}
Arguing as in the proof of Lemma~\ref{lem:mvsquarefree}, let $\cJ_{r,k}(V)$ denote the set of all $(v_1,\dots,v_{2r})$ such that 
$$v_1^{j}+\dots+v_r^{j}=v_{r+1}^{j}+\dots+v_{2r}^{j}, \quad 1\le j \le k, \quad 1\le v_i \le V.$$
Expanding the $2r$-th power in the definition of $W$ and interchanging summation and integration gives
\begin{align*}
W \le \sum_{(v_1,\dots,v_{2r})\in \cJ_{r,k}(V)}\left|\sum_{\lambda=1}^{q}\chi \left(\frac{(\lambda+v_1)\dots(\lambda+v_r)}{(\lambda+v_{r+1}\dots(\lambda+v_{2r})} \right)\right|.
\end{align*}
As in Lemma~\ref{lem:mvsquarefree} we break the $\cJ_{r,k}(V)$ into sets $\cJ_{r,k}'(V)$ and $\cJ_{r,k}''(V)$, where
\begin{align*}
\cJ_{r,k}'(V)&=\{ (v_1,\dots,v_{2r})\in \cJ_{r,k}(V) : \text{at least} \  r+1 \  \text{of the} \  v_i's \   \text{are distinct}\}, \\
\cJ_{r,k}''(V)&=\{ (v_1,\dots,v_{2r})\in \cJ_{r,k}(V) : (v_1,\dots,v_{2r})\not \in \cJ_{r,k}'(V) \},
\end{align*} 
so that
\begin{align*}
W\ll q^{n}V^{r}+\sum_{(v_1,\dots,v_{2r})\in \cJ_{r,k}'(V)}\left|\sum_{\lambda=1}^{q}\chi \left(\frac{(\lambda+v_1)\dots(\lambda+v_r)}{(\lambda+v_{r+1}\dots(\lambda+v_{2r})} \right)\right|.
\end{align*}
From~\cite[Theorem 2C', pg 43]{Sch}, we have if $(v_1,\dots,v_{2r})\in \cJ_{r,k}'(V)$ then
$$\sum_{\lambda=1}^{q}\chi \left(\frac{(\lambda+v_1)\dots(\lambda+v_r)}{(\lambda+v_{r+1}\dots(\lambda+v_{2r})} \right)\ll q^{n/2},$$
so that  
\begin{align*}
W\ll q^{n}V^{r}+\sum_{(v_1,\dots,v_{2r})\in \cJ_{r,k}'(V)}q^{n/2},
\end{align*}
and the result follows since $\# \cJ_{r,k}'(V)\le J_{r,d}(V).$
\end{proof}
\section{Multiplicative energy of certain sets}
The following follows from the proof of~\cite[Lemma 7]{FGS}.
\begin{lemma}
\label{cong}
Let $M,N,U,q$ be integers with
$$NU\le q,$$
and let \  $\cU$ denote the set
$$\cU=\{ \  1\le u \le U \  : \  (u,q)=1 \  \}.$$
Then the number of solutions to the congruence
$$n_1u_1\equiv n_2u_2 \mod q, \quad M<n_1,n_2 \le M+N, \quad u_1,u_2\in \cU$$
is bounded by 
 $NUq^{o(1)}.$
\end{lemma}
The following is due to Konyagin~\cite[Lemma~1]{Kon}.
\begin{lemma}
\label{lem:Konyagin}
Let $q$ be prime and let $\omega_1,\dots,\omega_n\in \mathbb{F}_{q^{n}}$ be a basis for $\mathbb{F}_{q^{n}}$ as a vector space over $\mathbb{F}_{q}$. Let $\cB_1$ and $\cB_2$ denote the boxes 
$$\cB_1=\{ h_1\omega_1+\dots+h_n\omega_n : 1 \le h_i \le H\},$$ 
$$\cB_2=\{ h_1\omega_1+\dots+h_k\omega_n : 1 \le h_i \le U\},$$
and suppose that $H, U \le p^{1/2}$. Then the number of solutions to the equation
$$x_1x_2\equiv x_3x_4, \quad x_1,x_3\in \cB_1, \quad x_2,x_4 \in \cB_2,$$
is  $\ll (UH)^{n}\log{q}$.
\end{lemma}
The following is due to Bourgain and Chang~\cite{BC}.
\begin{lemma}
\label{lem:BC}
Let $q$ be prime, $L_1(\mathbf{x}),\dots,L_n(\mathbf{x})$ be linear forms in $n$ variables which are linearly independent$\mod q$ and let $\cB_1$ and $\cB_2$ denote the boxes
$$\cB_1=\{ \mathbf{h}=(h_1,\dots,h_n) : 1 \le h_i \le H \},$$ 
$$\cB_2=\{ \mathbf{h}=(h_1,\dots,h_n) : 1 \le h_i \le  U \}.$$
Then if $H, U \le p^{1/2}$ the number of solutions to the system of congruences
$$L_i(\mathbf{x}_1)L_i(\mathbf{x}_2)\equiv L_i(\mathbf{x}_3)L_i(\mathbf{x}_4) \mod q, \quad  \mathbf{x}_1,\mathbf{x}_3 \in \cB_1, \ \ \mathbf{x}_2,\mathbf{x}_4 \in \cB_2, \ \  \ \  1\le i \le n,$$
is bounded by 
 $(NH)^np^{o(1)}.$
\end{lemma}

\section{Proof of Theorem~1}
We define the integers
\begin{equation}
\label{U,V}
U= \left \lfloor \frac{N}{q^{1/2(r-d(d+1)/4)}} \right \rfloor, \quad V=\lfloor q^{1/2(r-d(d+1)/4)} \rfloor,
\end{equation}
and the set
$$\cU=\{ \ 1\le u\le U : (u,q)=1 \},$$
so that
\begin{equation}
\label{U bound}
\#\cU=Uq^{o(1)}.
\end{equation}
By Lemma~\ref{lem:smooth} we have
\begin{align*}
& \left|\sum_{M<n\le M+N}\chi(n)e^{2\pi i F(n)} \right| 
 \le  \\ & \quad \quad \quad \quad \quad \quad \quad \quad  \frac{q^{o(1)}}{\# \cU V}\sum_{M-N<n\le M+N}\sum_{u \in \cU}\left|\sum_{1\le v \le V}\chi(n+uv)e^{2\pi i F(n+uv)}e^{2\pi i \alpha v}\right|,
\end{align*}
for some $\alpha \in \mathbb{R}$. Let
\begin{equation}
\label{W def}
W=\sum_{M-N<n\le M+N}\sum_{u \in \cU}\left|\sum_{1\le v \le V}\chi(n+uv)e^{2\pi i F(n+uv)}e^{2\pi i \alpha v}\right|,
\end{equation}
then since the polynomial $F$ has degree $d$, we see that
\begin{align*}
W&\le \sum_{M-N<n\le M+N}\sum_{u \in  \cU}\max_{(\alpha_1,\dots,\alpha_d)\in [0,1]^{d}}\left|\sum_{1\le v \le V}\chi(n+uv)e^{2\pi i(\alpha_1 v+\dots+\alpha_d v^d)}\right| \\
&=\sum_{\lambda=1}^{q}I(\lambda)\max_{(\alpha_1,\dots,\alpha_d)\in [0,1]^{d}}\left|\sum_{1\le v \le V}\chi(\lambda+v)e^{2\pi i(\alpha_1 v+\dots+\alpha_d v^d)}\right|,
\end{align*}
where $I(\lambda)$ denotes the number of solutions to the congruence
$$nu^{*}\equiv \lambda \pmod q, \quad  M-N<n\le M+N, \quad u\in  \cU.$$
For  $i=1,\dots, d$, let $$\delta_i=\frac{1}{4V^i},$$ and define the functions $\phi_{i}(v)$ by
$$1=\phi_i(v)\int_{-\delta_i}^{\delta_i}e^{2\pi i x v^{i}}dx,$$
so that for $1<v\le V$ we have
\begin{equation}
\label{eq:phi}
\phi_i(v)=\frac{2\pi iv^i }{\sin(2\pi \delta_i v^i)} \ll \frac{1}{\delta_i }\ll V^i.
\end{equation}
Let
\begin{align*}
\boldsymbol \alpha &= (\alpha_1,\dots \alpha_d), \ \ \ \mathbf{x}= (x_1,\dots,x_d), \ \ \ 
\mathbf{v}= (v,\dots, v^d),
\end{align*}
and let $\cC(\delta)$ denote the rectangle
$$[-\delta_1,\delta_1]\times \dots \times [-\delta_d,\delta_d],$$
then we have
\begin{align*}
W \le \sum_{\lambda=1}^{q}I(\lambda)\max_{\boldsymbol \alpha\in [0,1]^{d}}\left|\sum_{v \le V}\left(\prod_{i=1}^{d}\phi_i(v)\right)\int_{\cC(\delta)}\chi(\lambda+v)e^{2\pi i<\boldsymbol \alpha+\mathbf{x},\mathbf{v}>}d\mathbf{x}\right|, 
\end{align*}
where $< . \ ,\ . >$ dentoes the standard inner product on $\mathbb{R}^{d}$. Hence
\begin{align*}
W \le \sum_{\lambda=1}^{q}\int_{\cC(\delta)}I(\lambda)\max_{\boldsymbol \alpha\in [0,1]^{d}}\left|\sum_{v \le V}\left(\prod_{i=1}^{d}\phi_i(v)\right)\chi(\lambda+v)e^{2\pi i<\boldsymbol \alpha+\mathbf{x},\mathbf{v}>}\right|d\mathbf{x}.
\end{align*}
Two applications of the H\"{o}lder inequality give
\begin{align*}
|W|^{2r}&\ll \left(\prod_{i=1}^{d}\delta_i \right)^{2r-1}\left(\sum_{\lambda=1}^{q}I(\lambda) \right)^{2r-2}\left(\sum_{\lambda=1}^{q}I(\lambda)^2 \right) \times \\ & \quad \quad \left(\sum_{\lambda=1}^{q}\max_{\boldsymbol \alpha\in [0,1]^{d}}\int_{\cC(\delta)}\left|\sum_{v \le V}\left(\prod_{i=1}^{d}\phi_i(v)\right)\chi(\lambda+v)e^{2\pi i<\boldsymbol \alpha+\mathbf{x},\mathbf{v}>}\right|^{2r}d\mathbf{x} \right). 
\end{align*}
Since we have
$$\sum_{\lambda=1}^{q}I(\lambda)\ll UV,$$
and the term
$$\sum_{\lambda=1}^{q}I(\lambda)^2,$$
is equal to the number of solutions to the congruence
$$n_1u_1 \equiv n_2u_2 \mod q, \quad 1\le n_1,n_2 \le N, \quad u_1,u_2\in \cU,$$
we have by Lemma~\ref{cong}
$$\sum_{\lambda=1}^{q}I(\lambda)^2\le NUq^{o(1)},$$
so that
\begin{align*}
W & \le  \left(\prod_{i=1}^{d}\delta_i\right)^{2r-1}(NU)^{2r-1}q^{o(1)} \\ &  \quad \quad \quad \quad \times \left(\sum_{\lambda=1}^{q}\max_{\boldsymbol \alpha\in [0,1]^{d}}\int_{\cC(\delta)}\left|\sum_{v \le V}\left(\prod_{i=1}^{d}\phi_i(v)\right)\chi(\lambda+v)e^{2\pi i<\boldsymbol \alpha+\mathbf{x},\mathbf{v}>}\right|^{2r}d\mathbf{x} \right).
\end{align*}
Let 
$$W_1=\sum_{\lambda=1}^{q}\max_{\boldsymbol \alpha\in [0,1]^{d}}\int_{\cC(\delta)}\left|\sum_{v \le V}\left(\prod_{i=1}^{d}\phi_i(v)\right)\chi(\lambda+v)e^{2\pi i<\boldsymbol \alpha+\mathbf{x},\mathbf{v}>}\right|^{2r}d\mathbf{x},$$
so that by~\eqref{eq:phi}
\begin{equation}
\label{W_1}
W^{2r}\le  V^{-(2r-1)d(d+1)/2}(NU)^{2r-1}q^{o(1)}W_1.
\end{equation}
We have
\begin{align*}
W_1&=\sum_{\lambda=1}^{q}\max_{\boldsymbol \alpha\in [0,1]^{d}}\int_{\cC(\delta)+\boldsymbol \alpha}\left|\sum_{v \le V}\left(\prod_{i=1}^{d}\phi_i(v)\right)\chi(\lambda+v)e^{2\pi i<\mathbf{x},\mathbf{v}>}\right|^{2r}d\mathbf{x} \\
&\ll \sum_{\lambda=1}^{q}\int_{[0,1]^d}\left|\sum_{v \le V}\left(\prod_{i=1}^{d}\phi_i(v)\right)\chi(\lambda+v)e^{2\pi i<\mathbf{x},\mathbf{v}>}\right|^{2r}d\mathbf{x}.
\end{align*}
By~\eqref{eq:phi}, for each $1\le v \le V$ we have
$$\prod_{i=1}^{d}\phi_i(v)\ll V^{d(d+1)/2},$$
hence by Lemma~\ref{lem:mvsquarefree} 
\begin{align*}
W_1
 \ll & V^{rd(d+1)}\left(qV^r+q^{1/2}V^{2r-d(d+1)/4})\right)q^{o(1)},
\end{align*}
so that by~\eqref{W def}
$$\left|\sum_{M<n\le M+N}\chi(n)e^{2\pi i F(n)}\right|^{2r}\le \frac{(NU)^{2r-1}\left(qV^r+q^{1/2}V^{2r-d(d+1)/4}\right)V^{d(d+1)/2}q^{o(1)}}{U^{2r}V^{2r}}.$$
Recalling the choices of $U$ and $V$ gives
$$\left|\sum_{M<n\le M+N}\chi(n)e^{2\pi i F(n)}\right|^{2r}\le N^{2r-2}q^{1/2+d(d+1)/8(r-d(d+1)/4)+1/2(r-d(d+1)/4)+o(1)}.$$
\section{Proof of Theorem~\ref{thm:main2}}
Let 
$$U= \left \lfloor \frac{N}{q^{1/1/2(r-d(d+1)/2)}} \right \rfloor, \quad V=\lfloor q^{1/2(r-d(d+1)/2)} \rfloor,$$
and let $\phi_i(v)$ be defined as in the proof of Theorem~\ref{thm:main1}. Then following the proof of Theorem~\ref{thm:main1} we have
$$\left|\sum_{M<n\le M+N}\chi(n)e^{2\pi i F(n)}\right|^{2r}=\frac{V^{-(2r-1)d(d+1)/2}(NU)^{2r-1}}{V^{2r}U^{2r}}W_1q^{o(1)},$$
where
$$W_1=\sum_{\lambda=1}^{q}\int_{[0,1]^d}\left|\sum_{v \le V}\left(\prod_{i=1}^{d}\phi_i(v)\right)\chi(\lambda+v)e^{2\pi i(x_1v+\dots+x_dv^d)}\right|^{2r}d\mathbf{x}.$$
By Lemma~\ref{lem:mvsquarefree1} we have
$$W_1\ll V^{rd(d+1)}(qV^{r}+qV^{2r-d(d+1)/2})q^{o(1)},$$
so that
$$\left|\sum_{M<n\le M+N}\chi(n)e^{2\pi i F(n)}\right|^{2r}\le V^{d(d+1)/2}\frac{(NU)^{2r-1}}{U^{2r}V^{2r}}\left(qV^{r}+q^{1/2}V^{2r-d(d+1)/2}\right)q^{o(1)},$$
so that recalling the choice of $U,V$ gives
$$\left|\sum_{M<n\le M+N}\chi(n)e^{2\pi i F(n)}\right|^{2r}\le N^{2r-2}q^{(r+1-d(d+1)/2)/2(r-d(d+1)/2)+o(1)} .$$
\section{Proof of Theorem~\ref{thm:main3}}
Let 
$$U= \left \lfloor \frac{N}{q^{n/2(r-d(d+1)/2)}} \right \rfloor, \quad V=\lfloor q^{n/2(r-d(d+1)/2)} \rfloor$$
and let $\cU$ denote the box
$$\cU=\{ u_1\omega_1+\dots+u_n\omega_n :  0<u_i \le U \}.$$
Then with notation as in Lemma~\ref{lem:smooth} we have
\begin{align*}
\left|\sum_{\mathbf{x}\in \cB}\chi(\mathbf{x})e^{2\pi i F(\mathbf{x})} \right|\le \frac{q^{o(1)}}{VU^n}\sum_{\mathbf{x}\in \cB_0}\sum_{\mathbf{u}\in \cU}\left|\sum_{1\le v \le V}\chi(\mathbf{x}+\mathbf{u}v)e^{2\pi i F(\mathbf{x}+\mathbf{u}v)+2\pi i \alpha v} \right|.
\end{align*}
Let 
\begin{equation}
\label{eq:W1}
W=\sum_{\mathbf{x}\in \cB_0}\sum_{\mathbf{u}\in \cU}\left|\sum_{1\le v \le V}\chi(\mathbf{x}+\mathbf{u}v)e^{2\pi i F(\mathbf{x}+\mathbf{u}v)+2\pi i \alpha v} \right|,
\end{equation}
so that expanding $F(\mathbf{x}+\mathbf{u}v)$ as a polynomial in $v$ gives
$$F(\mathbf{x}+\mathbf{u}v)=\sum_{i=0}^{d}F_i(\mathbf{x},\mathbf{u})v^{i},$$ 
for some real numbers $F_i(\mathbf{x},\mathbf{u})$. Hence we have
\begin{align*}
W&\le \sum_{\mathbf{x}\in \cB_0}\sum_{\mathbf{u}\in \cU}\max_{(\alpha_1,\dots,\alpha_d)\in \mathbb{R}^d}\left|\sum_{1\le v \le V}\chi(\mathbf{x}+\mathbf{u}v)e^{2\pi i(\alpha_1v+\dots+\alpha_dv^d)} \right| \\
&=\sum_{\lambda\in \mathbb{F}_{q^n}}I(\lambda)\max_{(\alpha_1,\dots,\alpha_d)\in \mathbb{R}^d}\left|\sum_{1\le v \le V}\chi(\lambda+v)e^{2\pi i(\alpha_1v+\dots+\alpha_dv^d)} \right|,
\end{align*}
where $I(\lambda)$ denotes the number of solutions to the equation in $\mathbb{F}_{q^{n}}$
 $$\mathbf{x}\mathbf{u}^{-1}=\lambda, \quad \mathbf{x} \in \cB_0, \quad \mathbf{u} \in \cU.$$
With $\phi_i(v)$, $\delta_i$ and $\cC(\delta)$ as in Theorem~\ref{thm:main1} and $\mathbf{v}=(v,\dots,v^d)$ we have
\begin{align*}
W\le \sum_{\lambda\in \mathbb{F}_{q^n}}\int_{\cC(\delta)}I(\lambda)\max_{\mathbf{\alpha}\in \mathbb{R}^d}\left|\sum_{1\le v \le V}\prod_{i=1}^{d}\phi_i(v)\chi(\lambda+v)e^{2\pi i<\mathbf{\alpha}+\mathbf{y},\mathbf{v}>} \right|d\mathbf{y}.
\end{align*}
By two applications of H\"{o}lder's inequality, we get
\begin{align*}
W^{2r}&\le V^{-(2r-1)d(d+1)/2}\left(\sum_{\lambda\in\mathbb{F}_{q^n}}I(\lambda)\right)^{2r-2}\left(\sum_{\lambda \in \mathbb{F}_{q^n}}I(\lambda)^2\right) \\ & \quad \quad \quad \quad \quad \times \left(\sum_{\lambda \in \mathbb{F}_{q^n}}\int_{\cC(\delta)}\max_{\mathbf{\alpha}\in \mathbb{R}^d}\left|\sum_{1\le v \le V}\prod_{i=1}^{d}\phi_i(v)\chi(\lambda+v)e^{2\pi i<\mathbf{\alpha}+\mathbf{y},\mathbf{v}>} \right|^{2r}d\mathbf{y}\right).
\end{align*}
We have
$$\sum_{\lambda\in\mathbb{F}_{q^n}}I(\lambda)\ll (HU)^{n},$$
and the term
$$\sum_{\lambda\in\mathbb{F}_{q^n}}I(\lambda)^2,$$
is equal to the number of solutions to the equation over $\mathbb{F}_{q^n}$
$$\mathbf{x}_1\mathbf{u}_1=\mathbf{x}_2\mathbf{u}_2, \quad \mathbf{x}_1,\mathbf{x}_2 \in \cB_0, \quad \mathbf{u}_1,\mathbf{u}_2 \in \cU,$$
so that by Lemma~\ref{lem:Konyagin}
$$\sum_{\lambda\in\mathbb{F}_{q}}I(\lambda)^2\le (HU)^{n}q^{o(1)},$$
hence we get
\begin{align*}
W^{2r}&\le V^{-(2r-1)d(d+1)/2}(HU)^{(2r-1)n}q^{o(1)} \\ & \quad \quad \quad \quad  \times \left(\sum_{\lambda \in \mathbb{F}_{q}}\int_{\cC(\delta)}\max_{\mathbf{\alpha}\in \mathbb{R}^d}\left|\sum_{1\le v \le V}\prod_{i=1}^{d}\phi_i(v)\chi(\lambda+v)e^{2\pi i<\mathbf{\alpha}+\mathbf{y},\mathbf{v}>} \right|^{2r}d\mathbf{y}\right).
\end{align*}
Let $$W_1=\sum_{\lambda \in \mathbb{F}_{q}}\int_{\cC(\delta)}\max_{\mathbf{\alpha}\in \mathbb{R}^d}\left|\sum_{1\le v \le V}\prod_{i=1}^{d}\phi_i(v)\chi(\lambda+v)e^{2\pi i<\mathbf{\alpha}+\mathbf{y},\mathbf{v}>} \right|^{2r}d\mathbf{y},$$
so that 
\begin{align*}
W_1&=\sum_{\lambda \in \mathbb{F}_{q}}\max_{\mathbf{\alpha}\in \mathbb{R}^d}\int_{\cC(\delta)+\mathbf{\alpha}}\left|\sum_{1\le v \le V}\prod_{i=1}^{d}\phi_i(v)\chi(\lambda+v)e^{2\pi i<\mathbf{y},\mathbf{v}>} \right|^{2r}d\mathbf{y} \\
&\ll \sum_{\lambda \in \mathbb{F}_{q}}\int_{[0,1]^d}\left|\sum_{1\le v \le V}\prod_{i=1}^{d}\phi_i(v)\chi(\lambda+v)e^{2\pi i<\mathbf{y},\mathbf{v}>} \right|^{2r}d\mathbf{y}, 
\end{align*}
hence by Lemma~\ref{lem:mv2} we have
$$W_1\ll V^{rd(d+1)}\left(q^{n}V^{r}+q^{n/2}V^{2r-d(d+1)/2}\right),$$
which gives
$$\left|\sum_{\mathbf{x}\in \cB}\chi(\mathbf{x})e^{2\pi i F(\mathbf{x})} \right|^{2r}\le V^{d(d+1)/2}\frac{(HU)^{(2r-1)n}}{V^{2r}U^{2rn}}\left(q^{n}V^{r}+q^{n/2}V^{2r-d(d+1)/2}\right)q^{o(1)},$$
Recalling the choices of $U,V$ we get
$$\left|\sum_{\mathbf{x}\in \cB}\chi(\mathbf{x})e^{2\pi i F(\mathbf{x})} \right|^{2r}\le H^{(2r-2)n}q^{(nr-nd(d+1)/2)/2(r-d(d+1)/2)+o(1)}  .$$

\section{Proof of Theorem~\ref{thm:main4}}
Let $q=q_1\dots q_n$ and define the integers
$$V=\lfloor q^{1/2(r-d(d+1)/2)} \rfloor, U_i= \left \lfloor \frac{H_i}{q^{1/2(r-d(d+1)/2)}}\right \rfloor,$$
and the box
$$\cU=\{ (u_1,\dots,u_n): 1\le u_i \le U_i \},$$
so that by Lemma~\ref{lem:smooth} we have for some $\alpha\in \mathbb{R}$
\begin{align*}
\left|\sum_{\mathbf{x}\in \cB}\chi_1(x_1)\dots \chi_n(x_n)e^{2\pi i F(\mathbf{x})} \right|& \le \frac{q^{o(1)}}{VU_1\dots U_n}\sum_{\mathbf{x}\in \cB_0}\sum_{\mathbf{u}\in \cU} \\ & \quad \quad \left|\sum_{1\le v\le V}\chi_1(x_1+u_1v)\dots \chi_n(x_n+u_nv)e^{2\pi i( F(\mathbf{x}+\mathbf{u}v)+\alpha v)} \right|.
\end{align*}
Writing 
$$W=\sum_{\mathbf{x}\in \cB_0}\sum_{\mathbf{u}\in \cU} \left|\sum_{1\le v\le V}\chi_1(x_1+u_1v)\dots \chi_n(x_n+u_nv)e^{2\pi i( F(\mathbf{x}+\mathbf{u}v)+\alpha v)} \right|,$$
and letting $I(\lambda_1,\dots,\lambda_n)$ denote the number of solutions to the system of congruences
$$x_iu_i^{-1}\equiv \lambda_i \mod q_i, \quad N_i-H_i<x_i \le N_i+H_i, 1\le u_i \le U_i, \quad 1\le i \le n,$$
we see that
\begin{align*}
W&\le \sum_{\substack{\lambda_i=1 \\ 1\le i \le n}}^{q_i}I(\lambda_1,\dots,\lambda_n)\left|\sum_{1\le v\le V}\chi_1(\lambda_1+v)\dots \chi_n(\lambda_n+v)e^{2\pi i( F(\mathbf{x}+\mathbf{u}v)+\alpha v)} \right| \\
&\le \sum_{\substack{\lambda_i=1 \\ 1\le i \le n}}^{q_i}I(\lambda_1,\dots,\lambda_n) \max_{\alpha_1,\dots, \alpha_d}\left|\sum_{1\le v\le V}\chi_1(\lambda_1+v)\dots \chi_n(\lambda_n+v)e^{2\pi i( \alpha_1v+\dots+\alpha_dv^{d})} \right|. 
\end{align*}
With notation as in the proof of Theorem~\ref{thm:main1}, we see that
\begin{align*}
W&\le \sum_{\substack{\lambda_i=1 \\ 1\le i \le n}}^{q_i}\int_{\cC(\delta)}I(\lambda_1,\dots,\lambda_n)\max_{\boldsymbol \alpha\in [0,1]^{d}} \\ & \quad \quad \quad \times \left|\sum_{1\le v \le V}\left(\prod_{i=1}^{d}\phi_i(v)\right)\chi_1(\lambda_1+v)\dots \chi_n(\lambda_n+v)e^{2\pi i<\boldsymbol \alpha+\mathbf{x},\mathbf{v}>}\right|d\mathbf{x}.
\end{align*}
Two applications of H\"{o}lder's inequality give
\begin{align*}
W^{2r}\le V^{-(2r-1)d(d+1)/2}\left(\sum_{\substack{\lambda_i=1 \\ 1\le i \le n}}^{q_i}I(\lambda_1,\dots,\lambda_n) \right)^{2r-2}\left(\sum_{\substack{\lambda_i=1 \\ 1\le i \le n}}^{q_i}I(\lambda_1,\dots,\lambda_n)^2 \right)W_1,
\end{align*}
where 
\begin{align*}
W_1=\sum_{\substack{\lambda_i=1 \\ 1\le i \le n}}^{q_i}\int_{\cC(\delta)}\max_{\boldsymbol \alpha\in [0,1]^{d}}\left|\sum_{1\le v \le V}\left(\prod_{i=1}^{d}\phi_i(v)\right)\chi_1(\lambda_1+v)\dots \chi_n(\lambda_n+v)e^{2\pi i<\boldsymbol \alpha+\mathbf{x},\mathbf{v}>}\right|^{2r}d\mathbf{x}.
\end{align*}

As in the proof of Theorem~\ref{thm:main1} we have
$$W_1\ll \sum_{\substack{\lambda_i=1 \\ 1\le i \le n}}^{q_i}\int_{[0,1]^d}\left|\sum_{1\le v \le V}\left(\prod_{i=1}^{d}\phi_i(v)\right)\chi(\lambda+v)e^{2\pi i(x_1v+\dots+x_dv^d)}\right|^{2r}d\mathbf{x},$$
hence by Lemma~\ref{lem:mvmulti}
$$W_1\le V^{rd(d+1)}\left(qV^r+q^{1/2}V^{2r-d(d+1)/2}\right)q^{o(1)}.$$
We have
$$\sum_{\substack{\lambda_i=1 \\ 1\le i \le n}}^{q_i}I(\lambda_1,\dots,\lambda_n)\ll H_1\dots H_nU_1\dots U_n,$$
and the term
$$\sum_{\substack{\lambda_i=1 \\ 1\le i \le n}}^{q_i}I(\lambda_1,\dots,\lambda_n)^2,$$
is equal to the number of solutions to the system of equations
$$x_{i,1}u_{i,1}\equiv x_{i,2}u_{i,2} \mod q_i, \quad N_i-H_i<x_{i,1},x_{i,2} \le N_i+H_i, 1\le u_{i,1},u_{i,2} \le U_i, \quad 1\le i \le n,$$
hence by Lemma~\ref{cong} we have
$$\sum_{\substack{\lambda_i=1 \\ 1\le i \le n}}^{q_i}I(\lambda_1,\dots,\lambda_n)^2\le H_1\dots H_nU_1\dots U_nq^{o(1)},$$
which gives
\begin{align*}
\left|\sum_{\mathbf{x}\in \cB}\chi_1(x_1)\dots \chi_n(x_n)e^{2\pi i F(\mathbf{x})} \right|^{2r}\le V^{d(d+1)/2}\frac{(H_1\dots H_n)^{2r-1}}{V^{2r}U_1\dots U_n}\left(qV^r+q^{1/2}V^{2r-d(d+1)/2} \right)q^{o(1)}.
\end{align*}
Recalling the choices of $V,U_1, \dots, U_n$, we get
$$\left|\sum_{\mathbf{x}\in \cB}\chi_1(x_1)\dots \chi_n(x_n)e^{2\pi i F(\mathbf{x})} \right|^{2r}\le (H_1\dots H_n)^{2r-2}q^{(r-d(d+1)/2+n)/2(r-d(d+1)/2)+o(1)}.$$
\section{Proof of Theorem~\ref{thm:main5}}
We define the integers
$$U= \left \lfloor \frac{N}{q^{1/2(r-d(d+1)/2)}} \right \rfloor, \quad V=\lfloor q^{1/2(r-d(d+1)/2)} \rfloor,$$
and let $\cU$ and let denote the box
$$\cU=\{ (u_1,\dots,u_n): 1\le u_i \le U \},$$
so that from Lemma~\ref{lem:smooth} we have
\begin{align*}
\left|\sum_{\mathbf{x}\in \cB}\chi\left(\prod_{i=1}^{n}L_i(\mathbf{x}) \right)e^{2\pi i F(\mathbf{x})} \right|\le \frac{q^{o(1)}}{VU^{n}}\sum_{\mathbf{x}\in \cB_0}\sum_{\mathbf{u}\in \cU}\left|\sum_{1\le v \le V}\chi\left(\prod_{i=1}^{n}L_i(\mathbf{x}+\mathbf{u}v) \right)e^{2\pi i F(\mathbf{x})+2\pi i \alpha v} \right|,
\end{align*}
and since each $L_i$ is linear this gives
\begin{align*}
&\left|\sum_{\mathbf{x}\in \cB}\chi\left(\prod_{i=1}^{n}L_i(\mathbf{x}) \right)e^{2\pi i F(\mathbf{x})} \right|\le \\  
& \quad \quad \quad \quad \quad \quad \quad \quad \frac{q^{o(1)}}{VU^{n}}\sum_{\mathbf{x}\in \cB_0}\sum_{\mathbf{u}\in \cU}\left|\sum_{1\le v \le V}\chi\left(\prod_{i=1}^{n}(L_i(\mathbf{x})L_i(\mathbf{u})^{-1}+v) \right)e^{2\pi i F(\mathbf{x})+2\pi i \alpha v} \right|.
\end{align*}
Let
$$W=\sum_{\mathbf{x}\in \cB_0}\sum_{\mathbf{u}\in \cU}\left|\sum_{1\le v \le V}\chi\left(\prod_{i=1}^{n}(L_i(\mathbf{x})L_i(\mathbf{u})^{-1}+v) \right)e^{2\pi i F(\mathbf{x})+2\pi i \alpha v} \right|,$$
and let $I(\lambda_1,\dots,\lambda_n)$ denote the number of solutions to the system of equations
$$L_i(\mathbf{x})L_i^{-1}(\mathbf{u})\equiv \lambda_i \mod q, \quad \mathbf{x}\in \cB_0, \quad \mathbf{u}\in \cU, \quad 1\le i \le n,$$
then we have from the techniques of the preceeding arguments
\begin{align*}
W^{2r}&\le V^{-(2r-1)d(d+1)/2}\left(\sum_{\lambda_i=1}^{q}I(\lambda_1,\dots,\lambda_n) \right)^{2r-2}\left(\sum_{\lambda_i=1}^{q}I(\lambda_1,\dots,\lambda_n)^2 \right) \\ & \quad \quad \times\left(\sum_{\lambda_i=1}^{q}\int_{[0,1]^d}\left|\sum_{1\le v \le V}\left(\prod_{i=1}^{d}\phi_i(v)\right)\chi\left((\lambda_1+v)\dots(\lambda_n+v)\right)e^{2\pi i(x_1v+\dots+x_dv^d)}\right|^{2r}d\mathbf{x}\right).
\end{align*}
We have
$$\sum_{\lambda=1}^{q}I(\lambda)\ll (HU)^n,$$
and by Lemma~\ref{lem:BC}
$$\sum_{\lambda=1}^{q}I(\lambda)^2\le (HU)^nq^{o(1)}.$$
By Lemma~\ref{lem:mvmulti}
\begin{align*}
&\sum_{\lambda=1}^{q}\int_{[0,1]^d}\left|\sum_{1\le v \le V}\left(\prod_{i=1}^{d}\phi_i(v)\right)\chi(\lambda+v)e^{2\pi i(x_1v+\dots+x_dv^d)}\right|^{2r}d\mathbf{x} \\ & \quad \quad \quad \quad \quad \quad \quad \quad \quad \quad \quad \quad  \ll V^{rd(d+1)}\left(q^{n}V^{r}+q^{n/2}V^{2r-d(d+1)/2}\right),
\end{align*}
so that by the above
\begin{align*}
\left|\sum_{\mathbf{x}\in \cB}\chi\left(\prod_{i=1}^{n}L_i(\mathbf{x}) \right)e^{2\pi i F(\mathbf{x})} \right|^{2r}\le V^{d(d+1)/2}\frac{H^{(2r-1)n}}{V^{2r}U^n}\left(q^{n}V^{r}+q^{n/2}V^{2r-d(d+1)/2} \right),
\end{align*}
Recalling the choice of $U$ and $V$ gives
$$\left|\sum_{\mathbf{x}\in \cB}\chi\left(\prod_{i=1}^{n}L_i(\mathbf{x}) \right)e^{2\pi i F(\mathbf{x})} \right|^{2r}\le
H^{(2r-2)n}q^{n(r-D+1)/2(r-D)+o(1)}.
 $$

\end{document}